\documentclass[a4j,12pt]{amsart}

\usepackage{amsmath}
\usepackage{amsbsy}
\usepackage{amssymb}
\usepackage{amscd}
\usepackage{euscript}

\makeatletter
\renewcommand{\c@equation}{\c@thm}
\makeatother

\usepackage[dvips]{graphicx, color}
\usepackage{slashbox}

\usepackage{url}

\newcommand{\bbq}{{\mathbb Q}}
\newcommand{\bbr}{{\mathbb R}}

\newcommand{\bbz}{{\mathbb Z}}

\newcommand{\ord}{{\operatorname{ord}}}

\newcommand{\Z}{\bbz}
\newcommand{\Q}{\bbq}
\newcommand{\R}{\bbr}

\newcommand{\beq}{\begin{equation}\begin{aligned}}
\newcommand{\eeq}{\end{aligned}\end{equation}}

\newcommand{\beqn}{\begin{eqnarray}}
\newcommand{\eeqn}{\end{eqnarray}}

\newcommand{\tors}{{\operatorname{tors}}}

\newcommand{\beqnn}{\begin{eqnarray*}}
\newcommand{\eeqnn}{\end{eqnarray*}}
\newcommand{\beqnna}{\begin{eqnarray}\begin{array}}
\newcommand{\eeqnna}{\end{array}\end{eqnarray}}
\newcommand{\bsub}{\begin{subarray}}
\newcommand{\esub}{\end{subarray}}

\newtheorem{thm}{Theorem}[section]
\newtheorem{lem}[thm]{Lemma}

\newtheorem{prop}[thm]{Proposition}

\theoremstyle{definition}

\newtheorem{rem}[thm]{Remark}        

\title{}

\date{}

\newcommand{\bQ}{\ensuremath{\mathbb{Q}}}

\title
[On the elliptic curve]
{The Mordell-Weil bases for the elliptic curve of the form 
$\boldsymbol{y^2=x^3-m^2x+n^2}$}
\author{Yasutsugu Fujita}
\author{Tadahisa Nara}

\address[Y. Fujita]
{College of Industrial Technology, Nihon University, 
2-11-1 Shin-ei, 
Narashino, Chiba 275--8576, Japan}
\email{fujita.yasutsugu@nihon-u.ac.jp}

\address[T. Nara]
{Faculty of Engineering, Tohoku-Gakuin University, 
1-13-1 Chuo, 
Tagajo, Miyagi 985-8537, Japan}
\email{sa4m19@math.tohoku.ac.jp}



\begin{document}

\begin{abstract}
Let $E_{m,n}$ be an elliptic curve over $\Q$ of the form $y^2=x^3-m^2x+n^2$, where $m$ and $n$ are positive integers. Brown and Myers showed that the curve $E_{1,n}$ has rank at least two for all $n$.
In the present paper, we specify the two points which can be extended to a basis for $E_{1,n}(\Q)$ under certain conditions described explicitly. 
Moreover, we verify a similar result for the curve $E_{m,1}$, which, however, gives a basis for the rank three part of $E_{m,1}(\Q)$. 
\end{abstract}

\maketitle

\section{Introduction}\label{sec:introduction}

Let $m,\,n$ be positive integers and $E_{m,n}$ the elliptic curve defined by
\[
y^2=x^3-m^2x+n^2.
\]
Brown and Myers (\cite{BM}) examined the curve $E_{1,n}$ and found that the group $E_{1,n}(\bQ)$ of rational points on $E_{1,n}$ over $\Q$ has rank at least two as far as $n \ge 2$. 
After that, the curve $E_{m,1}$ was studied by Antoniewicz (\cite{Antoniewicz}), 
who showed that the group $E_{m,1}(\Q)$ has rank at least two if $m \ge 2$ and has rank at least three if $m \ge 4$ with $m \equiv 0 \pmod 4$ or $m=7$, which partially gave an answer to the problem raised in (\cite{BM}). 
Both curves above were further investigated 
in Eikenberg's dissertation (\cite{Eikenberg}), 
where it was shown that the group $E_{1,n}(\Q (n))$ of $\Q (n)$-rational points is generated by the points $(0,n)$ and $(1,n)$ (\cite[Corollary 3.1.2]{Eikenberg}), and that the group $E_{m,1}(\Q (m))$ of $\Q (m)$-rational points is generated by the points $(0,1)$, $(m,1)$ and $(-1,m)$ (\cite[Theorem 5.1.1]{Eikenberg}). 
Note that $\Q (n)$ and $\Q (m)$ in the assertions above are function fields. 
For high rank curves of the forms $E_{1,n}$ and $E_{m,1}$, 
see Tadi\'c's papers (\cite{Tadic1}, \cite{Tadic2}). \par
Let $P_0=(0,n)$ and $P_{\pm 1}=(\pm m,n)$ be integral points 
on $E_{m,n}$. 
It is easy to see that these points satisfy the relation
\[
P_0+P_{+1}+P_{-1}=O.
\]
Denote by $\Delta_{m,n}$ the discriminant of $E_{m,n}$, which equals $-16(27n^4-4m^6)$. 
The purpose of the present paper is to determine the bases for $E_{1,n}(\bQ)$ and $E_{m,1}(\bQ)$ under certain conditions described explicitly.

%
%
%

\begin{thm}\label{thm:gen}
Let $m,\,n$ be coprime positive integers. 
Assume that 
the $p$-primary part of 
$\Delta_{m,n}$ is square-free for any prime $p>3$. 
\begin{itemize}
\item[(1)] If $m=1$ and $n \geq2$, 
then $\{ P_0,\ P_{-1}\}$ can be extended to a basis for 
$E_{m,n}(\bQ)$. 
\item[(2)] If $n=1$ and $m \geq 4$, 
then  $\{P_0,\ P_{-1},\ P_{2}\}$ 
can be extended to a basis for 
$E_{m,n}(\bQ)$, 
where $P_{2}=(-1,\ m) \in E_{m,1}(\Q)$. 
\end{itemize}
\end{thm}

\begin{rem}\label{rem:m=n=1}
For some particular cases with $n=1$ we have the following results: 
\begin{itemize}
\item
$E_{1,1}(\Q)=\langle P_{+1} \rangle$ and 
$P_0=-3P_{+1}$;
\item
$E_{2,1}(\Q)=\langle P_0,\ P_{-1} \rangle$; 
\item
$E_{3,1}(\Q)=\langle P_0,\ P_{2} \rangle$ and 
$P_{+1}=2 P_{2}$;
\item
$E_{7,1}(\Q)=\langle P_{0},\ P_{2},\ (-3,\ 11) \rangle$ and 
$P_{+1}=-2(-3,\ 11)$;
\item
$E_{24,1}(\Q)=\langle P_{0},\ P_{2},\ (-10,\ 69) \rangle$ and 
$P_{+1}=-2(-10,\ 69)$;
\end{itemize}
where $E_{7,1}$ and $E_{24,1}$ do not satisfy the above 
assumption about $\Delta_{m,n}$. 
\end{rem}

While Eikenberg used the theory of Mordell-Weil lattices (see \cite{Shioda}) to find the bases for $E_{1,n}(\Q(n))$ and $E_{m,1}(\Q(m))$, 
we appeal to explicit estimates of canonical heights 
to show Theorem \ref{thm:gen}. 
There are several literatures describing explicitly the bases for 
the Mordell-Weil groups of parametric families of elliptic curves $E$ 
over $\Q$ under the assumption that 
$E$ has rank two or three 
(see, e.g., \cite{duquesne1}, \cite{FT}, \cite{FN1a}, \cite{Fujita1},
\cite{Fujita2}, \cite{FN2a}). 
However, as far as we can see, Theorem \ref{thm:gen} 
is the first result giving the bases 
in the cases where the $j$-invariants of $E$ are not equal to $0$ or $1728$. 
Although in general it is needed in order to get better lower bounds 
for canonical heights 
(see Propositions \ref{prop:lower} and \ref{prop:lower2}), 
in case $n=1$ the assumption on $\Delta_{m,n}$ is crucial because, 
otherwise, the assertion does not hold for $m \in \{7,24\}$, 
as seen in Remark \ref{rem:m=n=1}. 
Furthermore, one can expect that almost all of $m$ or $n$ satisfy the assumption on $\Delta_{m,n}$. 
More precisely, assuming that the $abc$ conjecture is true, we can estimate the density of $n$ (resp.~$m$) satisfying the assumption on 
$\Delta_{1,n}$ (resp.~$\Delta_{m,1}$) in Theorem \ref{thm:gen}.


\begin{prop}
\label{prop:density}
For $x>0$ define 
\begin{align*}
&\mathcal{N}(x)
=\#\{ n \in (0, x] \ ;\ \text{the $p$-primary part of }
\Delta_{1,n} \text{ is square-free for any } p>3 \},\\
&\mathcal{M}(x)
=\#\{ m \in (0, x] \ ;\ \text{the $p$-primary part of }
\Delta_{m,1} \text{ is square-free for any } p>3 \}.
\end{align*}
Suppose the $abc$ conjecture is true. 
Then there exist constants 
$\kappa_1,\, \kappa_2>0.97$ such that 
\[
\mathcal{N}(x) \sim \kappa_1 x,\quad 
\mathcal{M}(x) \sim \kappa_2 x.
\]
\end{prop}

The organization of this paper is as follows. 
In Section \ref{sec:prelim}, 
we quote the results from \cite{BM} and \cite{Antoniewicz} 
which show that $E_{m,n}(\Q)$ is torsion-free 
and has rank at least two under the assumptions in Theorem \ref{thm:gen}. 
In Section \ref{sec:reduction}, 
we examine the reduction types and the $x$-intercepts of $E_{m,n}$, which are needed in computing the canonical heights in the following sections. 
Section \ref{sec:E_{1,n}} is devoted to prove Theorem \ref{thm:gen} (1). 
In Section \ref{sec:E_{m,1}}, we prove Theorem \ref{thm:gen} (2) 
and Proposition \ref{prop:density}.

%
%

\section{Preliminaries}
\label{sec:prelim}

First, 
we have the following 
proposition by 
Brown and Myers (\cite[Theorem 3]{BM}) 
and Antoniewicz (\cite[Theorem 2.3]{Antoniewicz}). 

\begin{prop}\label{prop:tors}
Assume that one of the following holds$:$
\begin{itemize}
\item $m=1$ and $n \geq1;$
\item $n=1$ and $m \geq1$. 
\end{itemize}
Then, $E_{m,n}(\bQ)_{\tors}=\{O\}$.
\end{prop}

Next, 
in view of Lemma 6 in \cite{BM} and Lemmas 3.1 and 3.9 in \cite{Antoniewicz}, 
we have the following proposition. 

\begin{prop}\label{prop:2div}
Assume that one of the following holds$:$
\begin{itemize}
\item $m=1$ and $n \geq2;$
\item $n=1$ and $m \not \in \{1,3,7,24\}$.
\end{itemize}
Then, $P_0,\,P_{+1},\,P_0+P_{+1} \not \in 2E_{m,n}(\bQ)$. 
In particular, the points $P_0$ and $P_{+1}$ 
are independent modulo $E_{m,n}(\bQ)_{\tors}$ 
\end{prop}

\section{Local study of the curve}
\label{sec:reduction}

\begin{lem}
\label{lem:minimal}
If $\gcd(m,n)=1$, 
then the Weierstrass equation 
\begin{align}
\label{eq:WE}
y^2=x^3-m^2x+n^2 
\end{align}
for $E_{m,n}$ is global minimal.
\end{lem}
\begin{proof}
In view of \cite[VII, Remark 1.1]{aec}, 
it suffices to show that at least one of 
$v_p(c_4)<4,\ v_p(c_6)<6$ and $v_p(\Delta)<12$ 
holds for every prime $p$. 
Now we have 
\[
c_4=2^4\cdot 3m^2,\ c_6=-2^5\cdot3^3n^2,\ 
\Delta=2^4(2^2m^6-3^3n^4). 
\]
If $p>3$, then either 
$v_p(c_4)<4$ or $v_p(c_6)<6$ always holds. 
If $p \in \{2, 3\}$, then $v_p(\Delta)<12$ always holds. 
\end{proof}

\begin{lem}
\label{lem:red-p}
If $\gcd(m,n)=1$, 
then for a prime $p>3$ the reduction type of $E_{m,n}$ at $p$ 
is ${\rm I}_k$ (the Kodaira symbol), where  $k=\ord_p(\Delta_{m,n})$. 
\end{lem}

\begin{proof}
There exists a minimal Weierstrass 
equation $y^2=x^3+a_4x+a_6$ for $E_{m,n}$ 
such that $a_4,\ a_6$ and the discriminant $\Delta$ are 
as described in the table of Exercise 4.47 in \cite{AECadv}. 
Since the equation $y^2=x^3-m^2x+n^2$ is also minimal, 
we can transform $y^2=x^3+a_4x+a_6$ to $y^2=x^3-m^2x+n^2$ 
by some $[1,r,s,t]$, 
where $[u,r,s,t]$ means the transformation 
\begin{align*}
x\mapsto u^2x+r,\quad y\mapsto u^3y+su^2x+t. 
\end{align*}
Then it turns out that $r=s=t=0$ and so definitively 
$a_4=-m^2,\ a_6=n^2$. 
Since if $p$ divides $\Delta_{m,n}$, then $p$ divides neither $m$ nor $n$, 
we see that the possible reduction type is ${\rm I}_k$ 
by the table of Exercise 4.47. 
\end{proof}

\begin{lem}
\label{lem:red3}
If $\gcd(m,n)=1$, 
then the reduction type of $E_{m,n}$ at $3$ 
is as follows: 
\begin{enumerate}
\item ${\rm I}_0$ if $m\not\equiv 0 \pmod 3$ 

\item ${\rm II}$ if $m \equiv 0 \pmod 3$ and $n\not\equiv \pm1 \pmod 9$

\item ${\rm III}$ otherwise

\end{enumerate}
\end{lem}

\begin{proof}
If $m\not\equiv 0 \pmod 3$, 
then $\Delta_{m,n}$ is not divisible by $3$ 
and we have ${\rm I}_0$. 

Next assume $m\equiv 0 \pmod 3$. 
Then $n\not\equiv 0 \pmod 3$ and $\ord_3(\Delta_{m,n})=3$. 
Now there exists a minimal Weierstrass 
equation $y^2=x^3+a_2x^2+a_4x+a_6$ for $E_{m,n}$ 
such that $a_2,\ a_4,\ a_6$ and the discriminant $\Delta$ are 
as described in the table of Exercise 4.48 in \cite{AECadv}. 
Since the equation $y^2=x^3-m^2x+n^2$ is also minimal, 
we can transform $y^2=x^3+a_2x^2+a_4x+a_6$ to $y^2=x^3-m^2x+n^2$ 
by some $[1,r,s,t]$. 
In particular, the discriminants of the two equations are the same. 
Then we have $a_2=-3r$. 
So since $a_2$ is divisible by $3$,  
the possible reduction type is ${\rm II}$ or ${\rm III}$ 
by the table.  
Transforming $y^2=x^3-m^2x+n^2$ by $[1,-1,0,0]$, 
we have the equation 
\[
y^2=x^3-3x^2-(m^2-3)x+m^2+n^2-1. 
\]
Note that $m^2+n^2-1$ is divisible by $3$, 
since $n\not\equiv 0 \pmod 3$. 
Further it turns out that 
$m^2+n^2-1$ is divisible by $9$ 
if and only if $n\equiv \pm1 \pmod 9$. 
Tate's algorithm (\cite[p. 366]{AECadv}) 
with the fact completes the proof. 
\end{proof}

\begin{lem}
\label{lem:red2}
The reduction type of $E_{m,n}$ at $2$ 
is as follows: 
\begin{enumerate}
\item ${\rm IV}$ if $n\equiv 1 \pmod 2$ 

\item ${\rm III}$ if $n \equiv 0 \pmod 2$ and $m\equiv 1 \pmod 2$
\end{enumerate}
\end{lem}

\begin{rem}
If $n \equiv m\equiv 0 \pmod 2$, then 
various reduction types are possible. 
\end{rem}
%
%
\begin{proof}
First assume $n \equiv 1,\ m\equiv 0 \pmod 2$. 
By transforming the equation (\ref{eq:WE}) by $[1,0,0,1]$ 
we have the equation 
\[
y^2+2y=x^3-m^2x+n^2-1 
\]
with the quantities 
\[
b_8=-m^4,\quad b_6=4n^2.
\]
Then $b_8\equiv 0 \pmod 8$ and $b_6\equiv 4 \pmod 8$, 
which indicate the type ${\rm IV}$ 
by Tate's algorithm (\cite[p. 366]{AECadv}). 

Next assume $n\equiv 1,\ m\equiv 1 \pmod 2$. 
By transforming the equation (\ref{eq:WE}) by $[1,1,1,1]$ 
we have the equation 
\[
y^2+2xy+2y=x^3+2x^2+(1-m^2)x-m^2+n^2 
\]
with the quantities 
\[
b_8=-m^4-6m^2+12n^2+3,\quad b_6=-4m^2+4n^2+4.
\]
Then $b_8\equiv -1-6+12+3\equiv 0 \pmod 8$ and 
$b_6\equiv -4+4+4\equiv 4 \pmod 8$, 
which indicate the type ${\rm IV}$ by Tate's algorithm. 

Assume $n\equiv 0,\ m\equiv 1 \pmod 2$. 
By transforming the equation (\ref{eq:WE}) by $[1,1,1,0]$ 
we have the equation 
\[
y^2+2xy=x^3+2x^2+(3-m^2)x-m^2 +n^2+ 1
\]
with the quantities 
\[
b_8=-m^4 - 6m^2 +12n^2+3.
\]
Then $b_8\equiv -1-6+3\equiv 4 \pmod 8$ 
which indicates the type ${\rm III}$ by Tate's algorithm. 
\end{proof}

The next lemma is related to bounds for 
the real components of $E_{m,n}$. 

\begin{lem}
\label{lem:shift2}
Put $f(x)=x^3-m^2x+n^2$ and  $l=n^{2/3}>0$. 
If $27{n}^{4}-4{m}^{6}>0$, then 
$f(x)$ has only one real root, 
which is bounded below by each of  
\[
-l\left(1+\frac{m^2}{3l^2}\right), \quad  
-m \left(1+\frac{l^3}{2m^3}\right).  
\]
If $27{n}^{4}-4{m}^{6}<0$, then 
$f(x)$ has three real roots $\alpha,\ \beta,\ \gamma$. 
Further if we assume $m/l \geq 3^{1/3}=1.4422\cdots$ 
and $\alpha <\beta<\gamma$, then 
we have the estimates    
\begin{align*}
-m \left(1+\frac{l^3}{2m^3}\right)<\alpha<0<\beta
<\frac{2l^3}{m^2}\leq m \left(1-\frac{l^3}{m^3}\right)<\gamma. 
\end{align*}
\end{lem}

\begin{proof}
It is widely known that 
the number of the real roots of a cubic polynomial 
depends on the sign of the discriminant $\Delta=-16(27n^4-4m^6)$ 
and so we only show about the bounds. 

We have 
\[
f\left( -l-\frac{m^2}{3l}\right)=-m^6/(27l^3)<0,\quad  
f\left(-m-\frac{l^3}{2m^2}\right)
=-\frac{{l}^{6}\,\left( 6\,{m}^{3}+{l}^{3}\right) }{8\,{m}^{6}}<0, 
\]
which gives a proof for the one-real-root case. 

Next we have 
\begin{align*}
&f\left(\frac{2l^3}{m^2} \right)
=
-\frac{{l}^{3}\,\left( {m}^{2}-2\,{l}^{2}\right) \,
\left( {m}^{4}+2\,{l}^{2}\,{m}^{2}+4\,{l}^{4}\right) }{{m}^{6}}
<0, \\ 
&f\left(m\left(1-\frac{l^3}{m^3}\right) \right)
=
-\frac{{l}^{3}\,\left( {m}^{6}-3\,{l}^{3}\,{m}^{3}+{l}^{6}\right) }
{{m}^{6}}
=
-\frac{{l}^{3}\,m^3 \left( {m}^{3}-3\,{l}^{3}\right)+{l}^{9} }
{{m}^{6}}<0, \\
&m\left(1-\frac{l^3}{m^3}\right)-\frac{2l^3}{m^2}
=\frac{{m}^{3}-3\,{l}^{3}}{{m}^{2}}
\geq 0
\end{align*}
by the assumption $m/l \geq 3^{1/3}$. 
Those facts with $f(0)=n^2>0$ give the estimates for 
$\alpha,\ \beta,\ \gamma$. 
%
\end{proof}

%
%
%
%
%

\begin{rem}
\label{rem:taylor}
The values $-l-m^2/(3l)$ and 
$-m-l^3/(2m^2)$ are obtained from 
the leading terms of 
the Taylor expansions around $m=0$ and $l=0$, respectively,  
of the explicit roots 
\[
x_0=
{\left( \frac{\sqrt{27\,{l}^{6}-4\,{m}^{6}}}{6\,\sqrt{3}}
-\frac{{l}^{3}}{2}\right) }^{\frac{1}{3}}
-{\left( \frac{\sqrt{27\,{l}^{6}-4\,{m}^{6}}}{6\,\sqrt{3}}
+\frac{{l}^{3}}{2}\right) }^{\frac{1}{3}}
%
%
%
\]
of $f(x)$. 
\end{rem}

\section{Generators for $E_{1,n}(\Q)$}
\label{sec:E_{1,n}}
In this section we consider 
the curve 
$E=E_{1,n} : y^2=x^3-x+n^2$ 
over $\Q$ and 
show that $\{P_0,\ P_{-1}\}$ can be extended to a basis 
for $E(\Q)$. 
Our method largely depends on estimates of the canonical height. 
We compute it through the decomposition into the sum of the local height 
functions. 
In this paper the definition of the local height function 
follows \cite[Chap. VI]{AECadv}.

\begin{prop}
\label{prop:lower}
Assume that $n\geq 27$ and that 
the $p$-primary part of 
$\Delta_{1,n}=-16(27n^4-4)$ is square-free for any $p>3$. 
Then for any rational non-torsion point $P \in E_{1,n}(\Q)$ we have 
\[
\hat{h}(P)>\frac 13 \log n -0.619. 
\]
\end{prop}

\begin{proof}
We denote the local height function on $E$ for a place $p$ 
by $\lambda_p$ and put $v_p(\cdot)=-\log |\cdot|_p$.

First we compute 
$\lambda_{\infty}$. 
To ease notation put $l=n^{2/3} \geq 9$. 
Further to use Tate's series 
we take the Weierstrass model 
$E' : y^2=f(x-2l-1/(3l))$, 
where $f(x)=x^3-x+n^2$.   
(Our local height function is independent of models.)
Then for any $Q'\in E'(\R)$ we have $x(Q')>l$ 
by Lemma \ref{lem:shift2}. 
By Tate's series 
we have 
\begin{align*}
\lambda_{\infty}(P) 
=\frac 12 \log|x(P')|+\frac 1 8 \sum_{k=0}^{\infty} 4^{-k}\log |z(2^k P')|
+\frac1{12}v_{\infty}(\Delta), 
\end{align*}
where  
$P'$ is the corresponding point on $E'$ to $P$, 
\\\\
{
\tiny 
$z(P)=
\cfrac{27\,{l}^{4}\,{x}^{4}+\left( -648\,{l}^{6}-162\,{l}^{4}-18\,{l}^{2}\right) \,{x}^{2}+\left( 1512\,{l}^{7}+432\,{l}^{5}+72\,{l}^{3}+8\,l\right) \,x-648\,{l}^{8}-108\,{l}^{6}+27\,{l}^{4}-6\,{l}^{2}-1}{27\,{l}^{4}\,{x}^{4}} 
$}\\\\
and $x=x(P)$. 
We can regard $z(P)$ as a function of the variable $x$ and $l$ 
in the domain $\mathcal{D} : 9 \leq l \leq x$ 
and denote it by $z(x,l)$. 
Here 
by using the Mathematica functions ``MaxValue'' 
and ``MinValue'' (\cite{Mathematica})  
we can 
evaluate the maximum and the minimum of $z(x,l)$ 
in $\mathcal{D}$ 
to 
$z(9, 9)=9.0745\cdots$ 
and $z(x_0, 9)=0.09801\cdots$, respectively, 
where 
$x_0=25.054819\cdots$. 
Therefore for $l \geq 9$ and $P\in E'(\R)$ we have 
\begin{align}
\label{z-bound}
0.098<z(P) < 9.075. 
\end{align}
(The upper bound is not necessary here, but used in Proposition \ref{prop:upper}.) 
So we have 
\[
\lambda_{\infty}(P)>\frac 12\log l + \log (0.098) 
\times \frac 18\sum_{k=0}^{\infty}4^{-k}
+\frac1{12}v_{\infty}(\Delta)
=\frac 13 \log n +\frac{\log(0.098)}6
+\frac1{12}v_{\infty}(\Delta). 
\]

Next we compute the local height for nonarchimedean places. 
Let $\psi_2$ and $\psi_3$ be the division polynomials of $E$. 
Explicitly we have 
\[
\psi_2=2y,\quad \psi_3=3x^4-6x^2+12n^2x-1.
\]
If $P$ reduces to a nonsingular point 
modulo 2, then  
\[
\lambda_2(P)=\frac 12 \log \max \{1,\ |x(P)|_2\} 
+\frac1{12}v_2(\Delta)\geq \frac1{12}v_2(\Delta). 
\]

Put $X=x(P)$ and $Y=y(P)$. 
Assume 
$P$ reduces to a singular point modulo 2.  
Then it is necessary that $v_2(X)=0$ 
since $v_2(f_x(X))=v_2(3X^2-1)> 0$ is needed, 
where $f(x)=x^3-x+n^2$. 
If $n\equiv 1\pmod2$, then $v_2(Y^2)=v_2(X^3-X+n^2)=0$ 
and by Lemma \ref{lem:red2} with \cite[p. 353, (32)]{sil1} 
we have 
\[
\lambda_2(P)=\frac 13 \log|\psi_2(P)|_2+\frac1{12}v_2(\Delta)
=\frac 13 \log|2Y|_2+\frac1{12}v_2(\Delta)
= -\frac1 3\log 2+\frac1{12}v_2(\Delta).
\]
Similarly 
if $n\equiv 0\pmod2$, then since 
$\psi_3(P)=3X^4-6X^2+12n^2X-1\equiv 3-6-1\equiv 4 \pmod8$, 
we have 
\[
\lambda_2(P)=\frac 18 \log|\psi_3(P)|_2+\frac1{12}v_2(\Delta)
= -\frac14\log 2+\frac1{12}v_2(\Delta).
\]

In any case 
\[
\lambda_2(P)\geq  -\frac1 3\log 2+\frac1{12}v_2(\Delta).
\]

%

For $p=3$ by Lemma \ref{lem:red3} 
the reduction type is ${\rm I}_0$ and so 
\[
\lambda_3(P)=\frac 12 \log \max \{1,\ |x(P)|_3\} 
+\frac1{12}v_3(\Delta)
\geq \frac1{12}v_3(\Delta).
\]

For $p>3$ by Lemma \ref{lem:red-p} with the assumption 
that the $p$-primary part of $27n^4-4$ is square-free, 
the reduction type is either ${\rm I_0}$ or ${\rm I_1}$. 
So $P$ always reduces to a nonsingular point modulo $p$ 
(e.g. \cite[p. 365]{AECadv}) 
and we have 
\[
\lambda_p(P)=\frac 12 \log \max \{1,\ |x(P)|_p\} 
+\frac1{12}v_p(\Delta)
\geq \frac1{12}v_p(\Delta). 
\]

Finally we have
\begin{align*} 
\hat{h}(P)
&=\sum_{p \leq \infty}\lambda_p(P)\\
&>\frac 13 \log n +\frac{\log(0.098)}6
-\frac 13 \log 2
+\sum_{p\leq \infty}\frac1{12}v_p(\Delta)\\
&>\frac 13 \log n -0.619. 
\end{align*}
\end{proof}

\begin{prop}
\label{prop:upper}
Let $P_0=(0,n)$ and $P_{- 1}=(-1, n)$ be 
integral points on $E_{1,n}$. 
Assume that $n>27$. 
%
%
Then 
\begin{align*}
\hat{h}(P_0)&<\frac 13 \log n +0.716,\\
\hat{h}(P_{-1})&<\frac 13 \log n +0.541. 
\end{align*}
\end{prop}

\begin{proof}
As in the proof of Proposition \ref{prop:lower}, 
we compute local heights of $P_0$ and $P_{-1}$. 

For  
$p={\infty}$  
again we take the model 
$E' : y^2=(x-2l-1/(3l))^3-(x-2l-1/(3l))+l^3$, 
where $l=n^{2/3}>9$.  
Then on $E'$ the points 
\[
P_0'=(2l+1/(3l),\ n),\quad P_{-1}'=(-1+2l+1/(3l),\ n)
\]
correspond to $P_0$ and  $P_{-1}$, 
respectively. 
By Tate's series with the bound (\ref{z-bound}) 
we have 
\begin{align*}
\lambda_{\infty}(P_0) 
&=\frac 12 \log|2l+1/(3l)|+\frac 1 8 \sum_{k=0}^{\infty} 4^{-k}z(2^k P'_0)
+\frac1{12}v_{\infty}(\Delta)\\
&<\frac 12 \log|(2+9^{-2}\cdot3^{-1})l|+\frac 1 8 \sum_{k=0}^{\infty} 4^{-k}
\cdot \log (9.075) 
+\frac1{12}v_{\infty}(\Delta)\\
&=\frac 13 \log n +\frac 1 2\log (2+9^{-2}\cdot3^{-1}) +\frac 1 6 \log (9.075) 
+\frac1{12}v_{\infty}(\Delta) 
\end{align*}
and 
\begin{align*}
\lambda_{\infty}(P_{-1}) 
&=\frac 12 \log|-1+2l+1/(3l)|+\frac 1 8 \sum_{k=0}^{\infty} 4^{-k}
\log |z(2^k P'_{-1})|
+\frac1{12}v_{\infty}(\Delta)\\
&<\frac 12 \log|2l|+\frac 1 8 \sum_{k=0}^{\infty} 4^{-k}\cdot \log (9.075) 
+\frac1{12}v_{\infty}(\Delta)\\
&=\frac 13 \log n +\frac 1 2\log 2 +\frac 1 6 \log (9.075) 
+\frac1{12}v_{\infty}(\Delta). 
\end{align*}

For $p=2$ 
since 
$v_2(x(P_0)) >0$ and $v_2(x(P_{- 1}))=0$, 
$P_0$ and $P_{-1}$ reduce to a nonsingular point and 
a singular point, respectively, modulo $2$. 
Further recall if singular, 
then the reduction type is ${\rm IV}$ or ${\rm III}$. 
So the same argument in the proof of Proposition \ref{prop:lower} shows that  
\begin{align*}
&\lambda_2(P_0)=\frac 12 \log \max \{1,\ |x(P_0)|_2\}+\frac1{12}v_2(\Delta)
=\frac1{12}v_2(\Delta),\\
&\lambda_2(P_{-1})
\leq -\frac14 \log 2+\frac1{12}v_2(\Delta). 
\end{align*}
For $p\geq 3$ 
%
%
we have the trivial bounds valid for any integral point: 
\[
\lambda_p(P_0)\leq  \frac1{12}v_p(\Delta),\quad 
\lambda_p(P_{-1})
\leq  \frac1{12}v_p(\Delta). 
\]

By summing them up, 
we have 
\begin{align*}
\hat{h}(P_0) 
&<\frac 13 \log n +\frac 1 2\log (2+9^{-2}\cdot3^{-1})  +\frac 1 6 \log (9.075) 
+\sum_{p\leq \infty}\frac1{12}v_{p}(\Delta) \\
&<\frac 13 \log n +0.716 
\end{align*}
and 
\begin{align*}
\hat{h}(P_{-1}) 
&<\frac 13 \log n +\frac 1 2\log 2 +\frac 1 6 \log (9.075) 
-\frac 1 4\log 2
+\sum_{p\leq \infty}\frac1{12}v_{p}(\Delta) \\
&<\frac 13 \log n +0. 541. 
\end{align*}

\end{proof}

\begin{thm}
\label{thm:E_1n}
Let $P_0=(0,n)$ and $P_{-1}=(-1, n)$ be 
integral points on $E_{1,n}$. 
Assume that $n>1$ and 
the $p$-primary part of 
$\Delta_{1,n}$ is square-free for any $p>3$. 
Then $\{P_0,\ P_{-1}\}$ can be extended to a basis for $E_{1,n}(\Q)$. 
\end{thm}

\begin{proof}
By Proposition \ref{prop:tors}, $E_{1,n}(\Q)$ is torsion-free 
and by Proposition \ref{prop:2div} if $n>1$, 
then $P_0$ and $P_{+1}$ are independent and so are 
$P_0$ and $P_{-1}$. 
Let $\nu$ be the index of the span of $P_0$ and $P_{-1}$ 
in $\Z G_1 + \Z G_2$, where $G_1$ and $G_2$ are points contained 
in a basis for $E_{1,n}(\Q)$
such that
$P_0,\ P_{-1} \in \Z G_1 + \Z G_2$. 
It is sufficient to show $\nu=1$. 
By Siksek's theorem (\cite{siksek1}) we have 
\begin{align*}
\nu \leq \frac {2}{\sqrt{3}} \frac{\sqrt{R(P_0,\ P_{-1})}}{\lambda}, 
\end{align*}
where $R(P_0,\ P_{-1})$ is the regulator of $\{P_0,\ P_{-1}\}$,
explicitly, 
\begin{align*}
R(P_0,\ P_{-1})
&=\hat{h}(P_0)\hat{h}(P_{-1})-\langle P_0,\ P_{-1} \rangle^2 \\
&=\hat{h}(P_0)\hat{h}(P_{-1})
-\frac14 \left(\hat{h}(P_0+P_{-1})-\hat{h}(P_0)-\hat{h}(P_{-1})\right)^2 
\end{align*}
and $\lambda$ is any positive lower bound of $\hat{h}$ for 
non-torsion points in $E_{1,n}(\Q)$. 
Hence by Propositions \ref{prop:lower} and \ref{prop:upper} 
we have 
\begin{align*}
\nu 
&\leq \frac {2}{\sqrt{3}} \frac{\sqrt{\hat{h}(P_0)\hat{h}(P_{-1})}}{\lambda}\\
&\leq \frac2{\sqrt{3}} 
\frac{\sqrt{\left(\frac13\log n +0.716\right)
\left(\frac13\log n +0.541\right)}}
{\frac1{3}\log n -0.619}
\end{align*}
for $n>27$. 
By calculation we see that 
the right hand side is 
less than $3$ for $n>66$ and 
less than $5$ for $n>19$, 
which imply $\nu=1$ for $n>66$, 
and $\nu=1 \text{ or } 3$ for $27 < n \leq 66$, 
respectively. (Note that $2 \nmid \nu$ by Proposition \ref{prop:2div}.) 
Now by using 
the Magma function ``DivisionPoints'' (\cite{Mag}) 
we can confirm that 
$P_0,\ P_{-1},\ P_0 \pm P_{-1} \not\in 3E(\Q)$ 
for $27 < n \leq 66$, 
which implies even $\nu=1$ 
for $27 < n \leq 66$. 
%
%
Finally for $1 < n \leq 27$ we can check that 
$\{P_0,\ P_{-1}\}$ can be extended to a basis 
by using the Magma function ``Generators''. 
Indeed, we can obtain a basis for each $n \leq 27$. 
Then all we have to do is to check that 
the ratio $R'/R$ is much less than four (and nonzero), 
where $R$ is the regulator of the given basis 
and $R'$ is the regulator of 
a set which consists of 
$P_0, P_{-1}$ 
and appropriate points in the given basis. 
\end{proof}

\begin{verbatim}
------------------------------------
//Magma codes to test the 3-divisibility  for n=27--66 

for n in [27..66] do
E:=EllipticCurve([0,0,0,-1,n^2]);
P0:=E![0,n,1];
P1:=E![-1,n,1];
Pp:=P0+P1;
Pm:=P0-P1;

if 
{
<DivisionPoints(P0,3), DivisionPoints(P1,3), 
DivisionPoints(Pp,3),DivisionPoints(Pm,3)>
} ne 
{
<[], [], [], []>
} then
print n;
end if;
end for;
------------------------------------
//Magma codes to compute R and R's for n=2--27 
//In those cases the rank is at most 4 and we can find 
//appropriate points such that 0 < R'/R < 4, 
//which means R'/R = 1. 

m:=1;
for n in [2..27] do 
E:=EllipticCurve([0,0,0,-m^2,n^2]);
P0:=E![0,n]; P1:=E![m,n]; 
G:=Generators(E);
r:=#G; G;

if r eq 2 then
Regulator(G);
Regulator([P0, P1]);

else if r eq 3 then
Regulator(G);
for i in [1..r] do
Regulator([P0, P1, G[i]]);
end for;

else if r eq 4 then
Regulator(G);
for i in [1..r] do 
for j in [1..i-1] do
Regulator([P0, P1, G[i], G[j]]);
end for;end for;
end if;end if;end if;
end for;
-----------------------------------
\end{verbatim}

\section{Generators for $E_{m,1}(\Q)$}
\label{sec:E_{m,1}}
From this section we consider 
the curve 
$E=E_{m,1} : y^2=x^3-m^2x+1$ 
over $\Q$. 
The argument is essentially the same as 
that for $E_{1,n}$. 
However, owing to a geometrical property, 
estimates of the canonical height are slightly easier. 

We use the following modified Tate's series for 
the computation of 
the local height function. 

\begin{lem}
\label{lem:mod-tates}
Let $E/\R$ be an elliptic curve 
\[
y^2+a_1xy+a_3y=x^3+a_2x^2+a_4x+a_6. 
\]
%
Assume that $x(Q)>0$ 
for any $Q$ in the connected component of $O$ in $E(\R)$. 
Then for any $P \in E(\R)\setminus E[2]$, 
the following convergent series gives 
the archimedean part of the local height function: 
\begin{align}
\label{eq:mod-tates}
\lambda_{\infty}(P) 
=\frac 18 \log |u(P)|
+\frac 18 \sum_{k=1}^{\infty} 4^{-k} \log |z(2^k P)|
+\frac1{12} v_{\infty}(\Delta), 
\end{align}
where 
\begin{align*}
&u(Q)=x^4(Q)-b_4x^2(Q)-2b_6x(Q)-b_8, \\
&z(Q)=u(Q)/x^4(Q). 
\end{align*}
\end{lem}
\begin{proof}
First note $u(Q)=0$ if and only if $x(2Q)=0$ since 
\[
x(2Q)=\frac {u(Q)}{4x^3(Q) + b_2x^2(Q) + 2b_4x(Q)+b_6} 
\]
whose numerator and denominator have no common roots
(\cite[p. 458]{AECadv}). 
Note also we have the equality 
\[
(2y(Q)+a_1x(Q)+a_3)^2=4x^3(Q) + b_2x^2(Q) + 2b_4x(Q)+b_6, 
\] 
whose value is nonzero if $Q \not\in E[2]$. 

%

Whether $x(P)=0$ or not, 
we can use the original series of Tate (\cite{tate-letter})  
for $2P$ 
under our assumption. 
So by the property of $\lambda_{\infty}$ 
(e.g. \cite[Ch. VI, Theorem 1.1]{AECadv}): 
\begin{align*}
\lambda_{\infty}(2P) 
=4\lambda_{\infty}(P) 
-\log|2y(P)+a_1x(P)+a_3|
-\frac 14 v_{\infty}(\Delta) 
\quad \text{for } P \in E(\R)\setminus E[2], 
\end{align*}
we have 
\begin{align*}
\lambda_{\infty}(P) 
&=\frac 14 \lambda_{\infty}(2P) 
+\frac 14 \log|2y(P)+a_1x(P)+a_3|
+\frac 1{16} v_{\infty}(\Delta) \\
&=\frac 14 \left(
\frac 12\log |x(2P)|+
\frac 18 \sum_{k=0}^{\infty} 4^{-k} \log |z(2^{k+1} P)|
+\frac1{12} v_{\infty}(\Delta)
\right) \\
&\qquad +\frac 14 \log|2y(P)+a_1x(P)+a_3|
+\frac 1{16} v_{\infty}(\Delta) \\
&=\frac 18 \log |u(P)|
+\frac 18 \sum_{k=1}^{\infty} 4^{-k} \log |z(2^k P)|
+\frac1{12} v_{\infty}(\Delta). 
\end{align*}

\end{proof}

The following fact is also used for 
estimates of the local height function. 

\begin{lem}
\label{lem:id}
Let $E$ be an elliptic curve 
defined by a simple form 
\[
y^2=x^3+a_2x^2+a_4x+a_6 
\]
and  let 
\begin{align*}
&k=3\,{x}^{2}+2\,a_2\,x+4\,a_4-{a_2}^{2}, \\
&l=9\,{x}^{3}+9\,a_2\,{x}^{2}
+( 21\,a_4-4\,{a_2}^{2} ) \,x+27\,a_6-2\,a_2\,a_4 
\end{align*}
be functions on $E$. 
Then the identity 
\[
16\, k\cdot \psi_3-4\, l\cdot \psi_2^2=\Delta 
\]
holds, 
where $\psi_3$ and $\psi_2$ are the division polynomials defined by  
\begin{align*}
&\psi_3=3\,{x}^{4}+4\,{a}_{2}\,{x}^{3}+6\,{a}_{4}\,{x}^{2}+12\,{a}_{6}\,x+4\,{a}_{2}\,{a}_{6}-{a}_{4}^{2}, \\
&\psi_2=2y, 
\end{align*}
regarded as functions on $E$. 
\end{lem}
\begin{proof}
The substitution 
\[\psi_2^2=4(x^3+a_2x^2+a_4x+a_6 )\]
and computation give the result. 
\end{proof}

\begin{prop}
\label{prop:lower2}
Assume that $m\geq 10$ and that 
the $p$-primary part of $\Delta_{m,1} =-16(27-4m^6)$ 
is square-free for any $p>3$. 
Then for any rational non-torsion point $P \in E_{m,1}(\Q)$ we have 
\[
\hat{h}(P)>\frac 12 \log m -0.509. 
\]
\end{prop}

\begin{proof}

By Lemma \ref{lem:mod-tates} for $P \in E(\R)\setminus E[2]$ 
we have 
\begin{align*}
\lambda_{\infty}(P)
=\frac 18 \log |u(P)|
+\frac 18 \sum_{k=1}^{\infty} 4^{-k} \log |z(2^k P)|
+\frac1{12} v_{\infty}(\Delta), 
\end{align*}
where 
\[
u(P)=(x(P)^2+m^2)^2-8x(P), \quad
z(2^k P)=\frac{(x(2^k P)^2+m^2)^2-8x(2^k P)}{x^4(2^k P)}.
\]
So we have the estimates 
\begin{align}
\label{eq:z-bound2} 
(x(P)^2+m^2-1)^2 &< u(P)<(x(P)^2+m^2)^2+8(m+1), \\
1< \left(1+\frac{m^2-1}{x^2(2^k P)}\right)^2
&< z(2^k P)
< \left(1+\frac{m^2}{x^2(2^k P)}\right)^2
<\left(1+\frac{m^2}{(m-1)^2}\right)^2, 
\label{eq:z-bound2-2} 
\end{align}
where the upper bounds are due to 
$x(P)>-(m+1)$ and $x(2^k P)>m-1$ 
implied by Lemma \ref{lem:shift2}.
(The upper bounds are for later use.) 
Hence 
\begin{align*}
\lambda_{\infty}(P)
&> \frac 18 \log (m^2-1)^2 
+\frac1{12} v_{\infty}(\Delta) \\
&= \frac 14 \log \left\{ m^2 \left(1-\frac 1{m^2}\right) \right\}
+\frac1{12} v_{\infty}(\Delta) \\
&> \frac 14 \log \left\{ m^2 \left(1-\frac 1{10^2}\right) \right\}
+\frac1{12} v_{\infty}(\Delta) \\
&=\frac 12 \log m +\frac14 \log \frac {99}{100} +\frac1{12} v_{\infty}(\Delta). 
\end{align*}

%

Next we compute the local height for nonarchimedean places. 
Let 
\[
\psi_2=2y,\quad 
\psi_3=3\,{x}^{4}-6\,{m}^{2}\,{x}^{2}+12\,x-{m}^{4}
\]
be the division polynomials of $E$.
%
%
Put $X=x(P)$ and $Y=y(P)$.

We claim that 
\begin{align}
\label{eq:lower2m}
\lambda_2(P)
\begin{cases}
\geq \frac1{12}v_2(\Delta)\quad 
\text{if } P\ \text{reduces to a nonsingular point modulo 2,}\\
=-\frac13 \log 2 
+\frac1{12}v_2(\Delta)\quad 
\text{otherwise.}
\end{cases}
\end{align}

Indeed, 
if nonsingular it is clear and  
we assume $P$ reduces to a singular point modulo 2.  
Then by Lemma \ref{lem:red2} 
the reduction type is ${\rm IV}$ and so 
\[
\lambda_2(P)=\frac 13 \log |\psi_2(P)|_2 
+\frac1{12}v_2(\Delta)
=\frac 13 \log |2Y|_2 
+\frac1{12}v_2(\Delta).
\]
If $v_2(m)>0$, then $v_2(X)>0$ since $v_2(3X^2-m^2)>0$. 
So $v_2(Y^2)=v_2(X^3-m^2X+1)=0$. 
If $v_2(m)=0$, then $v_2(X)=0$ since $v_2(3X^2-m^2)>0$. 
So $v_2(Y^2)=v_2(X^3-m^2X+1)=0$. 
In any case $v_2(Y)=0$ and we have 
\[
\lambda_2(P)=-\frac13 \log 2 
+\frac1{12}v_2(\Delta).
\]

Similarly we claim that 
\begin{align}
\label{eq:lower3m}
\lambda_3(P) \geq 
\begin{cases}
\frac1{12}v_3(\Delta)\quad 
\text{if } P\ \text{reduces to a nonsingular point modulo 3,}\\
-\frac14 \log 3 
+\frac1{12}v_3(\Delta)\quad 
\text{otherwise.}
\end{cases}
\end{align}

Indeed, 
if nonsingular it is clear and  
we assume $P$ reduces to a singular point modulo 3.  
Then it is necessary that $v_3(Y)>0$ 
and $v_3(m)>0$ since 
$\frac{\partial}{\partial x} (x^3-m^2x+1-y^2)=3x^2-m^2$ 
and $\frac{\partial}{\partial y} (x^3-m^2x+1-y^2)=-2y$. 
Further note $v_3(X) \geq 0$ since $v_3(3X^2-m^2)>0$. 
Now the reduction type is {III} by Lemma \ref{lem:red3} 
and 
\[
\lambda_3(P)=\frac 18 \log|\psi_3(P)|_3+\frac1{12}v_3(\Delta). 
\]
By Lemma \ref{lem:id} we have the identity 
\[
16\,(3\,{X}^{2}-4\,{m}^{2})\cdot \psi_3(P) 
- 4\,(9\,{X}^{3}-21\,{m}^{2}\,X+27)\cdot \psi_2^2(P) 
=\Delta. 
\]
Note that $\ord_3 \Delta=3$ and 
\[
\ord_3((9\,{X}^{3}-21\,{m}^{2}\,X+27)\cdot \psi_2^2(P))\geq 4,
\quad \ord_3(3X^2-4m^2)\geq 1.  
\] 
This indicates $\ord_3 \psi_3(P) \leq 2$ and so 
\begin{align*}
\lambda_3(P)\geq \frac 18 \cdot (-2\log 3)+\frac1{12}v_3(\Delta)
=-\frac 14 \log 3+\frac1{12}v_3(\Delta). 
\end{align*}

For $p>3$ by Lemma \ref{lem:red-p} with the assumption 
that $27-4m^6$ has no square factor, 
the reduction type is either ${\rm I_0}$ or ${\rm I_1}$ and so 
\begin{align}
\label{eq:lowerpm}
\lambda_p(P)=\frac 12 \log \max \{1,\ |x(P)|_p\} 
+\frac1{12}v_p(\Delta)
\geq \frac1{12}v_p(\Delta). 
\end{align}

Finally we obtain 
\begin{align*}
\hat{h}(P)
=\sum_{p \leq \infty}\lambda_p(P)
>\frac 12 \log m +\frac14 \log \frac{99}{100}
-\frac13 \log 2 -\frac 14\log 3 
>\frac 12 \log m -0.509. 
\end{align*}
\end{proof}

\begin{prop}
\label{prop:upper2}
Let $P_0=(0,1)$, $P_{-1}=(-m, 1)$ 
and $P_2=(-1, m)$ 
be 
integral points on $E_{m,1}$. 
Assume that $m \geq 10$. 
%
%
Then 
for 
$P \in \{P_{0},\ P_{-1},\ P_{2},\ P_{0}+P_{-1},\ 
P_{-1}+P_{2},\ P_{0}+P_{-1}+P_{2}\}$ 
\[
\hat{h}(P)< \frac 12 \log m +0.290. 
\]
Further if we assume $p$-primary part of 
$\Delta_{m,1}$ is square-free for any $p>3$, 
then 
\begin{align*}
\log m 
-0.634 
<\hat{h}(P_{2}+P_{0})&<\log m 
+0.068.  
\end{align*}
\end{prop}
\begin{proof}
First we have the explicit expressions 
\begin{align*}
&P_{0}+P_{-1}=-P_{+1}=(m,\,-1),\\
&P_{-1}+P_{2}=(m+2,\,-2m-3),\\ 
&P_{2}+P_{0}=(m^2-2m+2,\,m^3-3m^2+4m-3),\\
&P_{0}+P_{-1}+P_{2}=(-m+2,\,-2m+3).  
\end{align*}
So we have 
%
%
%
%
\begin{align*}
&u(P_{0})=m^4,\quad
u(P_{-1})=4m^4+8m,\quad 
u(P_{2})=m^4+2m^2+9,\\
&u(P_{0}+P_{-1})=4m^4-8m,\quad
u(P_{-1}+P_{2})=4m^4+16m^3+32m^2+24m,\\
&u(P_{0}+P_{-1}+P_{2})= 4m^4-16m^3+32m^2-24m,
\end{align*}
where 
$u(Q)=(x(Q)^2+m^2)^2-8x(Q)$ 
as defined in Lemma \ref{lem:mod-tates}. 
Therefore for 
$P \in \{P_{0},\ P_{-1},\ P_{2},\ P_{0}+P_{-1},\ 
P_{-1}+P_{2},\ P_{0}+P_{-1}+P_{2}\}$ 
we have 
\begin{align*}
u(P) 
\leq 4m^4+16m^3+32m^2+24m 
&=4m^4(1+4/m+8/m^2+6/m^3) \\
&\leq 4m^4(1+4/10+8/10^2+6/10^3)
=4m^4\cdot \frac{743}{500}
\end{align*}

On the other hand, 
\begin{align*}
u(P_{2}+P_{0})=
{m}^{8}-8\,{m}^{7}+34\,{m}^{6}-88\,{m}^{5}
+153\,{m}^{4}-176\,{m}^{3}+128\,{m}^{2}-48\,m
\leq m^8
\end{align*}
for $m \geq 10$. 

By (\ref{eq:z-bound2-2}) 
we have 
\[
z(2^k Q)<\left(1+\frac{m^2}{(m-1)^2}\right)^2
\leq \left(1+\frac{10^2}{9^2}\right)^2
=\left(\frac{181}{81}\right)^2
\]
 for any $Q$. 
So for $P \in \{P_{0},\ P_{-1},\ P_{2},\ P_{0}+P_{-1},\ 
P_{-1}+P_{2},\ P_{0}+P_{-1}+P_{2}\}$,  
\begin{align*}
\lambda_{\infty}(P) 
&=\frac 18 \log|u(P)|+\frac 1 8 \sum_{k=1}^{\infty} 4^{-k}\log |z(2^k P)|
+\frac1{12}v_{\infty}(\Delta)\\
&<\frac 18 \log 4m^4\left(\frac{743}{500}\right)
+\frac 1 8 \sum_{k=1}^{\infty} 4^{-k}
\cdot \log \left(\frac{181}{81}\right)^2
+\frac1{12}v_{\infty}(\Delta)\\
&= \frac 12 \log m 
+\frac 18\log 4 
+\frac 1 {8}\log \frac{743}{500}
+\frac 1 {12}\log \frac{181}{81}
+\frac1{12}v_{\infty}(\Delta) \\
\end{align*}
and similarly, 
\begin{align*}
\lambda_{\infty}(P_{2}+P_{0}) 
&<\frac 18 \log m^8 
+\frac 1 8 \sum_{k=1}^{\infty} 4^{-k}
\cdot \log \left(\frac{181}{81}\right)^2
+\frac1{12}v_{\infty}(\Delta)\\
&=  \log m 
+\frac 1 {12}\log \frac{181}{81}
+\frac1{12}v_{\infty}(\Delta).
\end{align*}

Now since the relevant points are integral 
we clearly have 
\[
 \lambda_{p}(P)\leq \frac1{12}v_{p}(\Delta)
\]
for 
$p<\infty$. 

By summing up them, 
for 
$P \in \{P_{0},\ P_{-1},\ P_{2},\ P_{0}+P_{-1},\ 
P_{-1}+P_{2},\ P_{0}+P_{-1}+P_{2}\}$ 
\begin{align*}
\hat{h}(P) 
&<\frac 12 \log m 
+\frac 18\log 4 
+\frac 1 {8}\log \frac{743}{500}
+\frac 1 {12}\log \frac{181}{81}\\
&<\frac 12 \log m +0.290. 
\end{align*}
and 
\begin{align*}
\hat{h}(P_{2}+P_{0}) 
&< \log m  
+\frac 1 {12}\log \frac{181}{81}\\
&<\log m +0.068. 
\end{align*}

Next we shall obtain a lower bound for $\hat{h}(P_{2}+P_{0})$. 
By (\ref{eq:z-bound2}) and (\ref{eq:z-bound2-2}) we have 
\begin{align*}
\lambda_{\infty}(P_{2}+P_{0}) 
&\geq \frac 18 \log |x(P_{2}+P_{0})^2+m^2-1|^2 
+\frac1{12} v_{\infty}(\Delta) \\
&= \frac 14 \log (m^4-4m^3+9m^2-8m+3)
+\frac1{12} v_{\infty}(\Delta) \\
&> \frac 14 \log (m^4-4m^3)
+\frac1{12} v_{\infty}(\Delta) \\
&> \frac 14 \log \left\{ m^4 \left(1-\frac 4{10}\right) \right\}
+\frac1{12} v_{\infty}(\Delta) \\
&= \log m +\frac14 \log \frac {3}{5} +\frac1{12} v_{\infty}(\Delta). 
\end{align*}

Finally, with using (\ref{eq:lower2m}), (\ref{eq:lower3m}) 
and (\ref{eq:lowerpm}), 
we obtain 
\begin{align*}
\hat{h}(P_{2}+P_{0}) 
\geq \log m +\frac14 \log \frac {3}{5} -\frac13\log 2-\frac14\log 3
>\log m-0.634. 
\end{align*}
\end{proof}

\begin{thm}
\label{thm:E_m1}
Let $P_0=(0,1)$, $P_{-1}=(-m, 1)$ 
and $P_{2}=(-1,m)$ be 
integral points on $E_{m,1}$. 
Assume that $m >3$ and that 
the $p$-primary part of 
$\Delta_{m,1} =-16(27-4m^6)$ is square-free for any $p>3$. 
Then $\{P_0,\ P_{-1},\ P_{2}\}$ can be extended to a basis for $E_{m,1}(\Q)$. 
\end{thm}

\begin{proof}
Assume $m \geq 10$. 
By Proposition \ref{prop:tors}, $E_{m,1}(\Q)$ is torsion-free 
and by Proposition \ref{prop:2div} 
if $m \neq 24$, 
then $\{P_{0},\ P_{+1},\ P_{0}+P_{+1}\} \not\in 2E_{m,1}(\Q)$, 
equivalently, $\{P_{0},\ P_{-1},\ P_{0}+P_{-1}\} \not\in 2E_{m,1}(\Q)$. 
Further the facts 
\begin{align*}
&2^2\left(\frac 12 \log m -0.509\right)>\frac 12 \log m+0.290,\\
&2^2\left(\frac 12 \log m -0.509\right)>\log m +0.068
\end{align*}
with Propositions \ref{prop:lower2} and \ref{prop:upper2}
imply 
$\{P_{2},\ P_{-1}+P_{2},\ 
P_{2}+P_{0},\ P_{0}+P_{-1}+P_{2}\} \not\in 2E_{m,1}(\Q)$. 
Consequently $P_0,\ P_{-1},\ P_{2}$ are independent for $m\geq 10,\ m\neq 24$. 

Let $\nu$ be the index of the span of $P_0,\ P_{-1},\ P_{2}$ 
in $\Z G_1 + \Z G_2 + \Z G_3$, where $G_1,\ G_2,\ G_3$ are points contained 
in a basis for $E_{m,1}(\Q)$
such that
$P_0,\ P_{-1},\ P_{2} \in \Z G_1 + \Z G_2 + \Z G_3$. 
We should show $\nu=1$. 
%
%
%
%
%
%
%
%
%
%
%
%
First we estimate the height paring: 
\[
2\langle P_i,\, P_{j} \rangle
=\hat{h}(P_i+P_{j})-\hat{h}(P_i)-\hat{h}(P_{j}). 
\]
By Propositions \ref{prop:lower2} and \ref{prop:upper2} 
\begin{align*}
-\frac12 \log m -1.089&<2\langle P_0,\, P_{-1} \rangle
<\frac12 \log m +1.308,\\
-\frac12 \log m -1.089&<2\langle P_{-1},\, P_{2} \rangle
<\frac12 \log m +1.308, \\
-1.214&<2\langle P_{2},\, P_{0} \rangle
<1.086. 
\end{align*}

As the proof of Theorem \ref{thm:E_1n}, 
by Siksek's theorem 
\begin{align*}
\nu 
&\leq \sqrt{2}\,\sqrt{\frac{ R(P_0, P_{-1}, P_{2})}{\lambda^3}}.
\end{align*}
Since by definition 
\begin{align*}
R(P_0, P_{-1}, P_{2}) 
&=|\det\, (\langle P_i,\, P_{j} \rangle)_{i,j=0,-1,2}|\\
&< 
\hat{h}(P_0)\,\hat{h}(P_{-1})\,\hat{h}(P_2)
+2\langle P_0,\, P_{-1} \rangle 
\langle P_{-1},\, P_{2} \rangle 
\langle P_{2},\, P_{0} \rangle,  
\end{align*}
we have 
\begin{align*}
\sqrt{2}\,\sqrt{\frac{ R(P_0, P_{-1}, P_{2})}{\lambda^3}}
< \sqrt{2}\,\sqrt{\frac{(\frac12 \log m+{0.290})^3
+2(\frac14 \log m +\frac{1.308}2)^2\times \frac{1.214}2}
{(\frac12 \log m-0.509)^3}}
\end{align*}
for $m\geq 10$. 
By calculation we see that 
the right hand side is 
less than $3$ for $m\geq 59$, 
which imply $\nu=1$ for $m\geq 59$.  
(Note that $2 \nmid \nu$ by the above argument.) 
For $m < 59$ we can check that 
$\{P_0,\ P_{-1},\ P_{2}\}$ can be extended to a basis 
by using the Magma function ``Generators'' 
by the same manner as in Theorem \ref{thm:E_1n}. 
\end{proof}

\begin{rem}
During the check 
we can find that in the cases where $m=7,\ 24$, 
$\{P_0,\ P_{-1},\ P_{2}\}$ can not  be extended to a basis 
(in fact $\nu=2$), 
but in such cases 
the assumption that the $p$-primary part of $\Delta$ 
is square-free for $p>3$ is not satisfied.  
In the cases where $m=1,\ 2,\ 3$ the rank of $E_{m,1}(\Q)$ 
is less than three. 
\end{rem}

%
%
%
%

\begin{verbatim}
------------------------------------
//Magma codes to verify 0 < R'/R < 4 for m < 59 
 
n:=1;
for m in [4..59] do 
E:=EllipticCurve([0,0,0,-m^2,n^2]);
P0:=E![0,n]; P1:=E![m,n]; P2:=E![-1,m];
G:=Generators(E);
r:=#G; G;

if r eq 3 then
Regulator(G);
Regulator([P0, P1, P2]);

else if r eq 4 then
Regulator(G);
for i in [1..r] do
Regulator([P0, P1, P2, G[i]]);
end for;

else if r eq 5 then
Regulator(G);
for i in [1..r] do 
for j in [1..i-1] do
Regulator([P0, P1, P2, G[i], G[j]]);
end for;end for;
end if;end if;end if;
end for;
------------------------------------
\end{verbatim}

In the end of the paper 
we prove Proposition \ref{prop:density}. 
Before the proof we review the outline of 
the proof of \cite[Theorem 1]{granville}. 

For any polynomial $f$ such that 
$\{f(n)\ ; n\in \Z\}$ have no common square factor, 
put 
\begin{align*}
&S_1(x)=\#\{ n \in (0, x] \ ;\ p^2 \nmid f(n) 
\text{ for any } p \leq (\log x)/3 \},\\
&S_2(x)=\#\{ n \in (0, x] \ ;\ p^2 \mid f(n) \text{ for some }
 p \in ((\log x)/3 ,x] \},\\
&S_3(x)=\#\{ n \in (0, x] \ ;\ p^2 \mid f(n) \text{ for some } p>x \}.
\end{align*}
Then 
\begin{align*}
S_1(x)\sim \prod_{p} \left(1-\frac {\omega_f (p)}{p^2} \right)\cdot x,\quad
S_2(x)=o(x),\quad
S_3(x) =o(x), 
\end{align*}
where 
\[
\omega_f (p)
=\#\{n\pmod {p^2}\ ;\ f(n) \equiv 0 \pmod {p^2}\}. 
\]
The first estimate is due to the prime number theorem 
with the fact that 
the number of integers $n \in (a,\ a+\prod_{p \leq x_0}p^2]$ 
such that $p^2 \nmid f(n)$ for any $p \leq x_0$ 
is exactly 
\[
\prod_{p \leq x_0}p^2 \prod_{p\leq x_0} 
\left(1-\frac {\omega_f (p)}{p^2} \right) 
\]
independently of $a$, 
which is essentially from the Chinese remainder theorem.  
The estimate for $S_3$ needs the $abc$ conjecture. 
Consequently \cite[Theorem 1]{granville} is proved. 

So in our setting we have only to remove the factors 
$\left(1-\frac {\omega_f (2)}{2^2} \right)$ 
and 
$\left(1-\frac {\omega_f (3)}{3^2} \right)$ 
from the first estimate 
since we allow the discriminants $\Delta_{1,n}$ and  
$\Delta_{m,1}$ 
to be divisible by the square of $2$ or $3$, 
which does not alter the estimate for $S_2$ and $S_3$.

\begin{proof}[Proof of Proposition \ref{prop:density}]
\label{prf:density}
Put $D_{m,n}=27n^4-4m^6$, so that $\Delta_{m,n}=-16D_{m,n}$. 
Further define  
\begin{align*}
&\Omega_1 (p)=\{n\pmod {p^2}\ ;\ D_{1,n} \equiv 0 \pmod {p^2}\},\\
&\Omega_2 (p)=\{m\pmod {p^2}\ ;\ D_{m,1} \equiv 0 \pmod {p^2}\}.
\end{align*}
Since the discriminants of 
$D_{1,n}$ and $D_{m,1}$ 
(as polynomials in $n$ and $m$, respectively) 
have no prime divisor other than $2$ or $3$, 
we have 
$\omega_1 (p):=\#\Omega_1 (p) \leq 4$ 
and 
$\omega_2 (p):=\#\Omega_2 (p) \leq 6$ 
for $p>3$ 
by \cite[Lemma 5.2]{murty-pasten}. 
%
%
%

In view of the argument just before the proof 
we see that 
\begin{align*}
\kappa_1
=\prod_{p>3} \left(1-\frac {\omega_1 (p)}{p^2} \right)
=\prod_{k=3}^{\infty} \left(1-\frac {\omega_1 (p_k)}{p_k^2} \right), 
\end{align*}
where $p_k$ is the $k$-th prime number. 
Now by using the inequality 
\[
\prod_{k=1}^{N} \left(1-a_k\right) \geq 1-\sum_{k=1}^{N} a_k
\]
for $0<a_k<1$, which can be seen by induction, 
we have 
\begin{align*}
\kappa_1
=\prod_{k=3}^{60} \left(1-\frac {\omega_1 (p_k)}{p_k^2} \right)
\prod_{k=61}^{\infty} \left(1-\frac {\omega_1 (p_k)}{p_k^2} \right)
&\geq \prod_{k=3}^{60} \left(1-\frac {\omega_1 (p_k)}{p_k^2} \right)
\prod_{k=61}^{\infty} \left(1-\frac {4}{p_k^2} \right) \\
&\geq \prod_{k=3}^{60} \left(1-\frac {\omega_1 (p_k)}{p_k^2} \right)
\left(1-\sum_{k=61}^\infty \frac {4}{p_k^2} \right) \\
&=0.972866\cdots \times 
 0.997939\cdots 
>0.97, 
\end{align*}
where we compute $\omega_1 (p_k)$ for $k\leq 60$ directly and 
use the known result of the prime zeta function (e.g. \cite[p. 95]{finch-mc}): 
\[
\sum_{k=1}^\infty \frac {1}{p_k^2}=0.4522474200\cdots.
\]
By the same argument 
\begin{align*}
\kappa_2
=\prod_{k=3}^{60} \left(1-\frac {\omega_2 (p_k)}{p_k^2} \right)
\prod_{k=61}^{\infty} \left(1-\frac {\omega_2 (p_k)}{p_k^2} \right)
&\geq \prod_{k=3}^{60} \left(1-\frac {\omega_2 (p_k)}{p_k^2} \right)
\prod_{k=61}^{\infty} \left(1-\frac {6}{p_k^2} \right) \\
&\geq \prod_{k=3}^{60} \left(1-\frac {\omega_2 (p_k)}{p_k^2} \right)
\left(1-\sum_{k=61}^\infty \frac {6}{p_k^2} \right) \\
&=0.976111\cdots \times 
 0.996909\cdots 
>0.97. 
\end{align*}
\end{proof}


\begin{verbatim}
--------------------------------------
//Pari/GP codes to compute the products 

/* let f in Z[x] */
omg(f,m)=
{
local(N);
N=0;
for(i=1,m,
if(Mod(subst(f,x,i),m)==Mod(0,m),
N=N+1;
);
);
N
}

prod(i=3,60, 1-omg(27*x^4-4,prime(i)^2)/prime(i)^2 )*1.

1-4*0.4522474200410654985+4*sum(i=1,60, 1/prime(i)^2 )


prod(i=3,60, 1-omg(27-4*x^6,prime(i)^2)/prime(i)^2 )*1.

1-6*0.4522474200410654985+6*sum(i=1,60, 1/prime(i)^2 )
--------------
\end{verbatim}

\bibliographystyle{jplain}
\bibliography{RefM.bib}

\end{document}